\documentclass[12pt,reqno]{amsart}

\usepackage[arrow,matrix,curve]{xy}

\usepackage[dvips]{graphicx} 

\usepackage{amssymb, latexsym, amsmath, amscd, array,
}

\usepackage{amsfonts}     
\usepackage{tikz}     
\usepackage{scalefnt} 

\newtheorem{theorem}{Theorem}[section]

\newtheorem{proposition}[theorem]{Proposition}

\theoremstyle{definition}

\numberwithin{equation}{section}

\newcommand\Los{{\L}o{\'s}}

\newcommand\N {{\mathbb N}} 

\newcommand\R {{\mathbb R}}
\newcommand\Q {{\mathbb Q}}
 
\newcommand\st{{\rm st}} 

\newcommand\parisotes{{$\pi\alpha\rho\iota\sigma\acute{o}\tau\eta\varsigma$}}

\newcommand\Stromholm{{Str\o mholm}}

\newcommand{\halo}{{\scalebox{1}[.3]{\ensuremath{\bigcirc}}}} 

\newcommand{\halobig}{{\scalebox{2}[.3]{\ensuremath{\bigcirc}}}}

\author[T. B.]{Tiziana Bascelli} \address{T. Bascelli, Via
S. Caterina, 16, 36030 Montecchio P.no (VI), Italy}
\email{tiziana.bascelli@virgilio.it}

\author[E. B.]{Emanuele Bottazzi} \address{E. Bottazzi, Dipartimento
di Matematica, Universit\`a di Trento, Italy}
\email{Emanuele.Bottazzi@unitn.it}

\author[F. H.]{Frederik Herzberg} \address{F. Herzberg, Center for
Mathematical Economics, Bielefeld University, D-33615 Bielefeld,
Germany}\email{fherzberg@uni-bielefeld.de}

\author[V. K.]{Vladimir Kanovei} \address{V. Kanovei, IPPI, Moscow,
and MIIT, Moscow, Russia}\email{kanovei@googlemail.com}

\author[K. K.]{Karin U. Katz}\address{K. Katz, Department of
Mathematics, Bar Ilan University, Ramat Gan 52900
Israel}\email{katzmik@macs.biu.ac.il}

\author[M. K.]{Mikhail G. Katz}\address{M. Katz, Department of
Mathematics, Bar Ilan University, Ramat Gan 52900
Israel}\email{katzmik@macs.biu.ac.il}

\author[T. N.]{Tahl Nowik}\address{T. Nowik, Department of
Mathematics, Bar Ilan University, Ramat Gan 52900
Israel}\email{tahl@math.biu.ac.il}

\author[D. S.]{David Sherry}\address{D. Sherry, Department of
Philosophy, Northern Arizona University, Flagstaff, AZ 86011, US}
\email{David.Sherry@nau.edu}

\author[S. S.]{Steven Shnider}\address{S.~Shnider, Department of
Mathematics, Bar Ilan University, Ramat Gan 52900
Israel}\email{shnider@macs.biu.ac.il}

\begin{document}

\thispagestyle{empty}


\title[Fermat, Leibniz, Euler, and the gang] {Fermat, Leibniz, Euler,
and the gang: The true history of the concepts of limit and shadow}

\begin{abstract}
Fermat, Leibniz, Euler, and Cauchy all used one or another form of
approximate equality, or the idea of discarding ``negligible'' terms,
so as to obtain a correct analytic answer.  Their inferential moves
find suitable proxies in the context of modern theories of
infinitesimals, and specifically the concept of \emph{shadow}.  We
give an application to decreasing rearrangements of real functions.
\end{abstract}

\maketitle

\tableofcontents

\section{Introduction}
\label{one}

The theories as developed by European mathematicians prior to 1870
differed from the modern ones in that none of them used the modern
theory of limits.  Fermat develops what is sometimes called a
``precalculus" theory, where the optimal value is determined by some
special condition such as equality of roots of some equation.  The
same can be said for his contemporaries like Descartes, Huygens, and
Roberval.

Leibniz's calculus advanced beyond them in working on the derivative
function of the variable~$x$.  He had the indefinite integral whereas
his predecessors only had concepts more or less equivalent to it.
Euler, following Leibniz, also worked with such functions, but
distinguished the variable (or variables) with constant
differentials~$dx$, a status that corresponds to the modern assignment
that~$x$ is the independent variable, the other variables of the
problem being dependent upon it (or them) functionally.

Fermat determined the optimal value by imposing a condition using his
\emph{adequality} of quantities.  But he did not really think of
quantities as functions, nor did he realize that his method produced
only a necessary condition for his optimisation condition.  For a more
detailed general introduction, see chapters 1 and 2 of the volume
edited by Grattan-Guinness (Bos et al.~1980 \cite{Bo}).

The doctrine of \emph{limits} is sometimes claimed to have replaced
that of \emph{infinitesimals} when analysis was rigorized in the 19th
century.  While it is true that Cantor, Dedekind and Weierstrass
attempted (not altogether successfully; see Ehrlich 2006 \cite{Eh06};
Mormann \& Katz 2013 \cite{MK}) to eliminate infinitesimals from
analysis, the history of the limit concept is more complex.  Newton
had explicitly written that his \emph{ultimate ratios} were not
actually ratios but, rather, \emph{limits} of prime ratios (see
Russell 1903 \cite[item 316, p.~338-339]{Ru03}; Pourciau 2001
\cite{Pou}).  In fact, the sources of a rigorous notion of limit are
considerably older than the 19th century.

In the context of Leibnizian mathematics, the limit of~$f(x)$ as~$x$
tends to~$x_0$ can be viewed as the ``assignable part'' (as Leibniz
may have put it) of~$f(x_0+dx)$ where~$dx$ is an ``inassignable''
infinitesimal increment (whenever the answer is independent of the
infinitesimal chosen).  A modern formalisation of this idea exploits
the standard part principle (see Keisler 2012 \cite[p.~36]{Ke12}).

In the context of ordered fields~$E$, the \emph{standard part
principle} is the idea that if~$E$ is a proper extension of the real
numbers~$\R$, then every finite (or \emph{limited}) element~$x\in E$
is infinitely close to a suitable~$x_0\in\R$.  Such a real number is
called the \emph{standard part} (sometimes called the \emph{shadow})
of~$x$, or in formulas,~$\st(x)=x_0$.  Denoting by~$E_{\text{f}}$ the
collection of finite elements of~$E$, we obtain a map
\[
\st: E_{\text{f}} \to \R.
\]
Here~$x$ is called \emph{finite} if it is smaller (in absolute value)
than \emph{some} real number (the term \emph{finite} is immediately
comprehensible to a wide mathematical public, whereas \emph{limited}
corresponds to correct technical usage); an \emph{infinitesimal} is
smaller (in absolute value) than \emph{every} positive real; and~$x$
is \emph{infinitely close} to~$x_0$ in the sense that~$x-x_0$ is
infinitesimal.

\def\lineticktop#1#2{
  \draw[-] ({#1},3.075) node [anchor=south] {#2} -- (#1,2.925) ;
}
 
\def\linetickbot#1#2{
  \draw[-] ({#1},-0.075) node [anchor=north] {#2} -- (#1,0.075) ;
}
 
\def\microscope#1{
  \draw [thick,-         ] (#1 -0.25,2.15) -- (#1+0,3) -- (#1 +0.25,2.15) ;
  \draw [thick,fill=white] (#1 +0.00,1.50) circle (1) ;  
  \fill [white           ] (#1 -0.25,2.15) -- (#1+0,3) -- (#1 +0.25,2.15) ;
  \draw [thick,blue,->   ](#1-0.95,1.5) -- (#1+0.95,1.5) ;
}
 
\def\microtick#1#2#3#4#5#6{
  \draw [-,#5] (#1+0.375*#2,1.5+#3) node [anchor=#6,font=\small] {#4} -- (#1+0.375*#2,1.5-#3) ;
  \draw [red,thin] (#1+0.375*#2,1.3) to [out=-90,in=90] (#1+0.00,0.6);
}
 
\begin{figure}

\begin {tikzpicture} [scale=2]
  \draw[blue,thick,->]      (-1.1,3) -- (4.1,3) coordinate (x axis);
 
  \microscope    {1.41}
  \microtick     {1.41}{-0.2}{0.1}{$r$}{}{south}
  \microtick     {1.41}{ 0.7}{0.05}{\tiny $r+\beta$}{blue}{south}
  \microtick     {1.41}{ 2.1}{0.05}{\tiny $r+\gamma$}  {blue}{south}
  \microtick     {1.41}{-1.7}{0.05}{\tiny $r+\alpha$}    {blue}{south}
  \draw[red,->] (1.41,0.6) -- (1.41,0.325) node [anchor=west] {$\scriptstyle \operatorname{st}$} -- (1.41,0.15) ;
 
  \lineticktop{-1}{$-1$}
  \linetickbot{-1}{$-1$}
  \lineticktop{ 0}{$ 0$}
  \linetickbot{ 0}{$ 0$}
  \lineticktop{+1}{$ 1$}
  \linetickbot{+1}{$ 1$}
  \lineticktop{+2}{$ 2$}
  \linetickbot{+2}{$ 2$}
  \lineticktop{+3}{$ 3$}
  \linetickbot{+3}{$ 3$}
  \lineticktop{+4}{$ 4$}
  \linetickbot{+4}{$ 4$}
 
  \lineticktop{+1.41}{}
  \linetickbot{+1.41}{$ r$}
 
  \draw[->]      (-1.1,0) -- (4.1,0) coordinate (x axis);
\end {tikzpicture}

\caption{\textsf{The standard part function, st, ``rounds off" a
finite hyperreal to the nearest real number.  The function st is here
represented by a vertical projection.  An ``infinitesimal microscope"
is used to view an infinitesimal neighborhood of a standard real
number $r$, where $\alpha$, $\beta$, and $\gamma$ represent typical
infinitesimals.  Courtesy of Wikipedia.}}
\label{vertical}
\end{figure}

Briefly, the standard part function ``rounds off'' a finite element
of~$E$ to the nearest real number (see Figure~\ref{vertical}).

The proof of the principle is easy.  A finite element~$x\in E$ defines
a Dedekind cut on the subfield~$\R\subset E$ (alternatively,
on~$\Q\subset\R$), and the cut in turn defines the real~$x_0$ via the
usual correspondence between cuts and real numbers.  One sometimes
writes down the relation
\[
x\approx x_0
\]
to express infinite closeness.  

We argue that the sources of such a relation, and of the standard part
principle, go back to Fermat, Leibniz, Euler, and Cauchy.  Leibniz
would discard the \emph{inassignable} part of~$2x+dx$ to arrive at the
expected answer,~$2x$, relying on his \emph{law of homogeneity} (see
Section~\ref{lei}).  Such an inferential move is mirrored by a
suitable proxy in the hyperreal approach, namely the standard part
function.

Fermat, Leibniz, Euler, and Cauchy all used one or another form of
approximate equality, or the idea of discarding ``negligible'' terms.
Their inferential moves find suitable proxies in the context of modern
theories of infinitesimals, and specifically the concept of
\emph{shadow}.

The last two sections present an application of the standard part to
decreasing rearrangements of real functions and to a problem on
divergent integrals due to S.~Konyagin.

This article continues efforts in revisiting the history and
foundations of infinitesimal calculus and modern nonstandard analysis.
Previous efforts in this direction include Bair et al.~(2013
\cite{B11}); Bascelli (2014 \cite{Ba14}); B\l{}aszczyk et al.~(2013
\cite{BKS}); Borovik et al.~(2012 \cite{BJK}, \cite{BK}); Kanovei et
al.~(2013 \cite{KKM}); Katz, Katz \& Kudryk (2014 \cite{KKK}); Mormann
et al.~(2013 \cite{MK}); Sherry et al.~(2014 \cite{SK}); Tall et
al.~(2014 \cite{TK}).

\section{Methodological remarks}
\label{two}

To comment on the historical subtleties of judging or interpreting
past mathematics by present-day standards,%
\footnote{Some reflections on this can be found in (Lewis 1975
\cite{Lew}).}
note that neither Fermat, Leibniz, Euler, nor Cauchy had access to the
\emph{semantic} foundational frameworks as developed in mathematics at
the end of the 19th and first half of the 20th centuries.  What we
argue is that their \emph{syntactic} inferential moves ultimately
found modern proxies in Robinson's framework, thus placing a firm
(relative to ZFC)%
\footnote{The Zermelo--Fraenkel Set Theory with the Axiom of Choice.}
semantic foundation underneath the classical procedures of these
masters.  Benacerraf (1965 \cite{Be65}) formulated a related dichotomy
in terms of mathematical practice \emph{vs} mathematical ontology.

For example, the Leibnizian laws of continuity (see Knobloch 2002
\cite[p.~67]{Kn02}) and homogeneity can be recast in terms of modern
concepts such as the \emph{transfer principle} and the \emph{standard
part principle} over the hyperreals, without ever appealing to the
semantic content of the technical development of the hyperreals as a
\emph{punctiform} continuum; similarly, Leibniz's proof of the product
rule for differentiation is essentially identical, at the syntactic
level, to a modern infinitesimal proof (see Section~\ref{lei}).

\subsection{A-track and B-track}

The crucial distinction between syntactic and semantic aspects of the
work involving mathematical continua appears to have been overlooked
by R.~Arthur who finds fault with the hyperreal proxy of the
Leibnizian continuum, by arguing that the latter was
\emph{non-punctiform} (see Arthur 2013 \cite{Ar}).  Yet this makes
little difference at the syntactic level, as explained above.
Arthur's brand of the syncategorematic approach following Ishiguro
(1990 \cite{Is}) involves a reductive reading of Leibnizian
infinitesimals as \emph{logical} (as opposed to \emph{pure}) fictions
involving a hidden quantifier \`a la Weierstrass, ranging over
``ordinary'' values.  This approach was critically analyzed in (Katz
\& Sherry 2013 \cite{KS1}); (Sherry \& Katz 2013 \cite{SK}); (Tho 2012
\cite{Th}).

Robinson's framework poses a challenge to traditional historiography
of mathematical analysis.  The traditional thinking is often dominated
by a kind of Weierstrassian teleology.  This is a view of the history
of analysis as univocal evolution toward the radiant Archimedean
framework as developed by Cantor, Dedekind, Weierstrass, and others
starting around 1870, described as the \emph{A-track} in a recent
piece in these \emph{Notices} (see Bair et al.~2013 \cite{B11}).

Robinson's challenge is to point out not only the possibility, but
also the existence of a parallel Bernoullian%
\footnote{Historians often name Johann Bernoulli as the first
mathematician to have adhered systematically and exclusively to the
infinitesimal approach as the basis for the calculus.}
track for the development of analysis, or \emph{B-track} for short.
The B-track assigns an irreducible and central role to the concept of
infinitesimal, a role it played in the work of Leibniz, Euler, mature
Lagrange,%
\footnote{In the second edition of his \emph{M\'ecanique Analytique}
dating from 1811, Lagrange fully embraced the infinitesimal in the
following terms: ``Once one has duly captured the spirit of this
system [i.e., infinitesimal calculus], and has convinced oneself of
the correctness of its results by means of the geometric method of the
prime and ultimate ratios, or by means of the analytic method of
derivatives, one can then exploit the infinitely small as a reliable
and convenient tool so as to shorten and simplify proofs''.  See (Katz
\& Katz 2011 \cite{KK11b}) for a discussion.}
Cauchy, and others.

The caliber of some of the response to Robinson's challenge has been
disappointing.  Thus, the critique by Earman (1975 \cite{Ea}) is
marred by a confusion of second-order infinitesimals like~$dx^2$ and
second-order hyperreal extensions like~$^\ast{^\ast\R}$; see (Katz \&
Sherry 2013 \cite{KS1}) for a discussion.

Victor J. Katz (2014 \cite{Ka14}) appears to imply that a B-track
approach based on notions of infinitesimals or indivisibles is limited
to ``the work of Fermat, Newton, Leibniz and many others in the 17th
and 18th centuries''.  This does not appear to be Felix Klein's view.
Klein formulated a condition, in terms of the mean value theorem,%
\footnote{The Klein--Fraenkel criterion is discussed in more detail in
Kanovei et al.~(2013 \cite{KKM}).}
for what would qualify as a successful theory of infinitesimals, and
concluded:
\begin{quote}
I will not say that progress in this direction is impossible, but it
is true that none of the investigators have achieved anything positive
(Klein 1908 \cite[p.~219]{Kl08}).
\end{quote}
Klein was referring to the current work on infinitesimal-enriched
systems by Levi-Civita, Bettazzi, Stolz, and others.  In Klein's mind,
the infinitesimal track was very much a current research topic; see
Ehrlich (2006 \cite{Eh06}) for a detailed coverage of the work on
infinitesimals around~1900.

\subsection{Formal epistemology: Easwaran on hyperreals}

Some recent articles are more encouraging in that they attempt a more
technically sophisticated approach.  K.~Easwaran's study (2014
\cite{Ea14}), motivated by a problem in formal epistemology,%
\footnote{The problem is concerned with saving philosophical
Bayesianism, a popular position in formal epistemology, which appears
to require that one be able to find on every algebra of doxastically
relevant propositions some subjective probability assignment such that
only the impossible event ($\emptyset$) will be assigned an
initial/uninformed subjective probability, or credence, of~$0$.}
attempts to deal with technical aspects of Robinson's theory such as
the notion of \emph{internal set}, and shows an awareness of recent
technical developments, such as a \emph{definable hyperreal system} of
Kanovei \& Shelah (2004 \cite{KS04}).

Even though Easwaran, in the tradition of Lewis (1980 \cite{Lew80})
and Skyrms (1980 \cite{Sk80}), tries to engage seriously with the
intricacies of employing hyperreals in formal epistemology,%
\footnote{For instance, he concedes: ``And the hyperreals may also
help, as long as we understand that they do not tell us the precise
structure of credences.'' (Easwaran 2014 \cite{Ea14}, Introduction,
last paragraph).}
not all of his findings are convincing.  For example, he assumes that
physical quantities cannot take hyperreal values.%
\footnote{Easwaran's explicit premise is that ``All physical
quantities can be entirely parametrized using the standard real
numbers.''  (Easwaran 2014 \cite[Section~8.4, Premise~3]{Ea14}).}
However, there exist physical quantities that are not directly
observable.  Theoretical proxies for unobservable physical quantities
typically depend on the chosen mathematical model.  And not
surprisingly, there are mathematical models of physical phenomena
which operate with the hyperreals, in which physical quantities take
hyperreal values.  Many such models are discussed in the volume by
Albeverio et al.~(1986 \cite{AFHL86}).
 
For example, certain probabilistic laws of nature have been formulated
using hyperreal-valued probability theory.  The construction of
mathematical Brownian motion by Anderson (1976 \cite{A76}) provides a
hyperreal model of the botanical counterpart.  It is unclear why (and
indeed rather implausible that) an observer~A, whose degrees of belief
about botanical Brownian motion stem from a mathematical model based
on the construction of mathematical Brownian motion by Wiener (1923
\cite{W23}) should be viewed as being more rational than another
observer~B, whose degrees of belief about botanical Brownian motion
stem from a mathematical model based on Anderson's construction of
mathematical Brownian motion.%
\footnote{One paradoxical aspect of Easwaran's methodology is that,
despite his anti-hyperreal stance in (2014 \cite{Ea}), he \emph{does}
envision the possibility of useful infinitesimals in an earlier joint
paper (Colyvan \& Easwaran 2008 \cite{CE}), where he cites John Bell's
account (Bell's presentation of Smooth Infinitesimal Analysis in
\cite{Be06} involves a category-theoric framework based on
intuitionistic logic); but never the hyperreals.  Furthermore, in the
2014 paper he cites the \emph{surreals} as possible alternatives to
the real number--based description of the ``structure of physical
space" as he calls it; see Subsection~\ref{will} below for a more
detailed discussion.}

Similarly problematic is Easwaran's assumption that an infinite
sequence of probabilistic tests must necessarily be modeled by the set
of standard natural numbers (this is discussed in more detail in
Subsection~\ref{will}).  Such an assumption eliminates the possibility
of modeling it by a sequence of infinite hypernatural length.  Indeed,
once one allows for infinite sequences to be modeled in this way, the
problem of assigning a probability to an infinite sequence of coin
tosses that was studied in (Elga 2004 \cite{E04}) and (Williamson 2007
\cite{Wi07}) allows for an elegant hyperreal solution (Herzberg 2007
\cite{H06BJPS}).

Easwaran reiterates the common objection that the hyperreals are
allegedly ``non-constructive'' entities.  The bitter roots of such an
allegation in the radical constructivist views of E. Bishop have been
critically analyzed in (Katz \& Katz 2011 \cite{KK11d}), and
contrasted with the liberal views of the Intuitionist A.~Heyting, who
felt that Robinson's theory was ``a standard model of important
mathematical research'' (Heyting 1973 \cite[p.~136]{He73}).  It is
important to keep in mind that Bishop's target was \emph{classical
mathematics} (as a whole), the demise of which he predicted in the
following terms: 
\begin{quote} 
Very possibly classical mathematics will cease to exist as an
independent discipline (Bishop 1968 \cite[p.~54]{Bi68}).
\end{quote}

\begin{figure}
\includegraphics[height=2.2in]{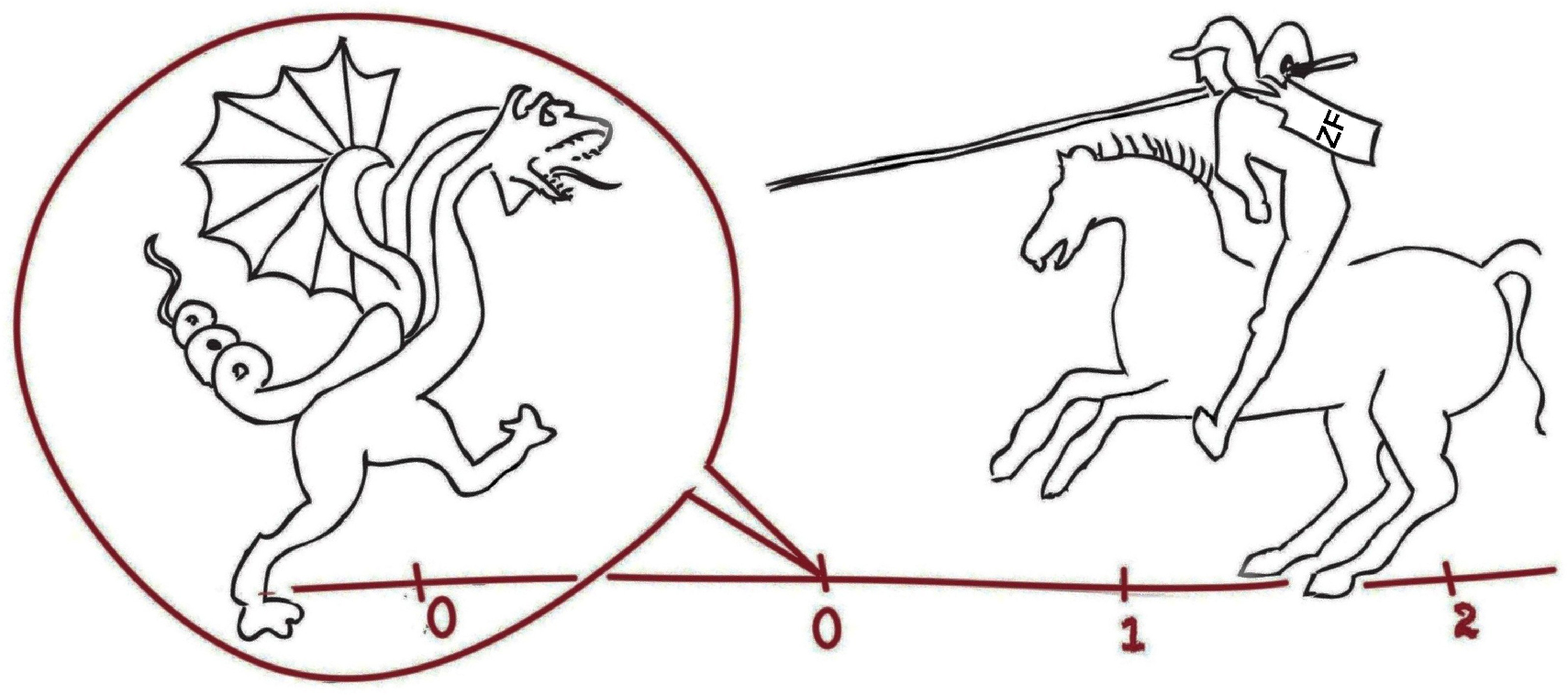}
%
%
\caption{\textsf{Easwaran's attempted slaying of the infinitesimal,
following P.~Uccello.  Uccello's creature is shown as inhabiting an
infinitesimal neighborhood of~$0$.}}
\label{uccello}
\end{figure}

\subsection{Zermelo--Fraenkel axioms and the Feferman--Levy model}

In his analysis, Easwaran assigns substantial weight to the fact that
``it is consistent with the ZF [Zermelo--Fraenkel set theory] without
the Axiom of Choice'' that the hyperreals do not exist (Easwaran 2014
\cite[Section~8.4]{Ea}); see Figure~\ref{uccello}.  However, on the
same grounds, one would have to reject parts of mathematics with
important applications.  There are fundamental results in functional
analysis that depend on the Axiom of Choice such as the Hahn--Banach
theorem; yet no one would suggest that mathematical physicists or
mathematical economists should stop exploiting them.

Most real analysis textbooks prove the~$\sigma$-additivity (i.e.,
countable additivity) of Lebesgue measure, but~$\sigma$-additivity is
not deducible from ZF, as shown by the Feferman--Levy model; see
(Feferman \& Levy 1963 \cite{FL}); (Jech 1973 \cite[chapter 10]{Je}).
Indeed, it is consistent with ZF that the following holds:
\begin{itemize}
\item[$(\ast)$] 
the continuum~$\R$ of real numbers is a countable union
$\R=\bigcup_{n\in\N}X_n$ of countable sets~$X_n$.
\end{itemize}
See (Cohen 1966 \cite[chapter IV, section 4]{Co66}) for a description
of a model of ZF in which~$(\ast)$ holds.%
\footnote{Property~$(\ast)$ may appear to be asserting the
countability of the continuum.  However, in order to obtain a
bijective map from a countable collection of countable sets to
$\N\times\N$ (and hence, by diagonalization, to~$\N$), the Axiom of
Choice (in its ``countable'' version which allows a countably-infinite
sequence of independent choices) will necessarily be used.}
Note that~$(\ast)$ implies that the Lebesgue measure is not countably
additive, as all countable sets are null sets whereas~$\R$ is not a
null set.  Therefore, countable additivity of the Lebesgue measure
cannot be established in ZF.

Terence Tao wrote:
\begin{quote}
By giving up countable additivity, one loses a fair amount of measure
and integration theory, and in particular the notion of the
expectation of a random variable becomes problematic (unless the
random variable takes only finitely many values).  (Tao 2013
\cite{Tao})
\end{quote}
Tao's remarks suggest that deducibility from ZF is not a reasonable
criterion of mathematical plausibility by any modern standard.

There are models of ZF in which there are infinitesimal numbers, if
properly understood, among the real numbers themselves.  Thus, there
exist models of ZF which are also models of Nelson's (1987 \cite{N87})
\emph{radically elementary} mathematics, a subsystem of Nelson's (1977
\cite{N77}) Internal Set Theory.  Here \emph{radically elementary}
mathematics is an extension of classical set theory (which may be
understood as ZF%
\footnote{Even though Nelson would probably argue for a much weaker
system; see Herzberg (2013 \cite[Appendix A.1]{Her}), citing Nelson
(2011 \cite{N11}).}
) by a unary predicate, to be interpreted as 
\begin{quote} 
``\ldots{} is a standard natural number", 
\end{quote}
with additional axioms that regulate the use of the new predicate
(notably external induction for standard natural numbers) and ensure
the existence of non-standard numbers. Nelson (1987
\cite[Appendix]{N87}) showed that a major part of the theory of
continuous-time stochastic processes is in fact equivalent to a
corresponding \emph{radically elementary} theory involving
infinitesimals, and indeed, \emph{radically elementary} probability
theory has seen applications in the sciences; see for example (Reder
2003 \cite{R03}).

In sum, mathematical descriptions of non-trivial natural phenomena
involve, by necessity, some degree of mathematical idealisation, but
Easwaran has not given us a good reason why only such mathematical
idealisations that are feasible in \emph{every} model of ZF should be
acceptable.  Rather, as we have already seen, there are very good
arguments (e.g., from measure theory) against such a high reverence
for ZF.

\subsection{Skolem integers and Robinson integers}

Easwaran recycles the well-known claim by A.~Connes that a
hypernatural number leads to a nonmeasurable set.  However, the
criticism by Connes%
\footnote{Note that Connes relied on the Hahn-Banach theorem,
exploited ultrafilters, and placed a nonconstructive entity (namely
the Dixmier trace) on the front cover of his \emph{magnum opus}; see
(Katz \& Leichtnam 2013 \cite{KL}) and (Kanovei et al.~2013
\cite{KKM}) for details.}
is in the category of dressing down a feature to look like a bug, to
reverse a known dictum from computer science slang.%
\footnote{See https://en.wikipedia.org/wiki/Undocumented\_feature}
This can be seen as follows.  The Skolem non-standard integers
$\N_{\text{Sko}}$ are known to be purely constructive; see Skolem
(1955 \cite{Sk55}) and Kanovei et al.~(2013 \cite{KKM}).  Yet they
imbed in Robinson's hypernaturals~$\N_{\text{Rob}}$:
\begin{equation}
\label{21}
\N_{\text{Sko}} \hookrightarrow \N_{\text{Rob}}.
\end{equation}
Viewing a purely constructive Skolem hypernatural
\[
         H \in \N_{\text{Sko}}\setminus \N
\]
as a member of~$\N_{\text{Rob}}$ via the inclusion~\eqref{21}, one can
apply the transfer principle to form the set
\[
         X_H = \{A \subset \N : H \in {}^\ast\!A\},
\]
where~${}^\ast\!A\subset \N_{\text{Rob}}$ is the natural extension
of~$A$.  The set $X_H$ is not measurable.  What propels the
set~$X_H\subset \mathcal{P}(\N)$ into existence is not a purported
\emph{weakness} of a nonstandard integer~$H$ itself, but rather the
remarkable \emph{strength} of both the \Los-Robinson transfer
principle and the consequences it yields.

\subsection{Williamson, complexity, and other arguments}
\label{will}

Easwaran makes a number of further critiques of hyperreal methodology.
His section 8.1, entitled ``Williamson's Argument", concerns infinite
coin tosses.  Easwaran's analysis is based on the model of a countable
sequence of coin tosses given by Williamson \cite{Wi07}.  In this
model, it is assumed that

\begin{quote}
\ldots{} for definiteness, [the coin] will be flipped once per second,
assuming that seconds from now into the future can be numbered with
the natural numbers (Easwaran 2014 \cite[section 8.1]{Ea14}).
\end{quote}
What is lurking behind this is a double assumption which, unlike other
``premises", is not made explicit by Easwaran.  Namely, he assumes
that

\begin{enumerate}
\item
a vast number of independent tests is best modeled by a temporal
arrangement thereof, rather than by a simultaneous collection; and
\item
the collection of seconds ticking away ``from now [and] into the
future" gives a faithful representation of the natural numbers.
\end{enumerate}

These two premises are not self-evident and some research
mathematicians have very different intuitions about the matter, as
much of the literature on applied nonstandard analysis (e.g.,
Albeverio et al.~1986 \cite{AFHL86}; Reder 2003 \cite{R03})
illustrates.

It seems that in Easwaran's model, an agent can choose not to flip the
coin at some seconds, thus giving rise to events like ``a coin that is
flipped starting at second 2 comes up heads on every flip".  However,
in all applications we are aware of, this additional structure used to
rule out the use of hyperreals as range of probability functions seems
not to be relevant.

Williamson and Easwaran appear to be unwilling to assume that, once
one decides to use hyperreal infinitesimals, one should also replace
the original algebra ``of propositions in which the agent has
credences" with an internal algebra of the hyperreal setting.  In
fact, such an additional step allows one to avoid both the problems
raised by Williamson's argument in his formulation using conditional
probability, and those raised by Easwaran in section 8.2 of his paper.

A possible model with hyperreal infinitesimals for an infinite
sequence of coin tosses is given by representing every event by means
of a sequence~$\{a_1, \ldots, a_N\}$, where~$a_n$ represents the
outcome of the~$n$th flip and~$N$ is a fixed hypernatural number.  In
this model, consider the events ``$a_n =$ Heads for~$n \leq N$'', that
we will denote~$H(1)$, and ``$a_n =$ Heads for~$2 \leq n \leq N$'',
that we will denote~$H(2)$.  In such a setting, events~$H(1)$ and
$H(2)$ are not isomorphic, contrary to what was argued in (Williamson
\cite[p.~3]{Wi07}).  This is due to the fact that hypernatural numbers
are an elementary extension of the natural numbers, for which the
formula~$k \not = k+1$ always holds.  Moreover, the probability of
$H(1)$ is the infinitesimal~$2^{-N}$, while the probability of~$H(2)$
is the strictly greater infinitesimal~$2^{-(N-1)}$, thus obeying the
well known rule for conditional probability.


Easwaran's section 8.4 entitled ``The complexity argument" is based on
four premises.  However, his premise 3, to the effect that ``all
physical quantities can be entirely parametrized using the standard
real numbers", is unlikely to lead to meaningful philosophical
conclusions based on ``first principles''.  This is because all
physical quantities can be entirely parametrized by the usual rational
numbers alone, due to the intrinsic limits of our capability to
measure physical quantities.  A clear explanation of this limitation
was given by Dowek.  In particular, since
\begin{quote}
a measuring instrument yields only an approximation of the measured
magnitude,
[\ldots] it is therefore impossible, except according to this
idealization, to measure more than the first digits of a physical
magnitude. [\ldots] According to this principle, this idealization of
the process of measurement is a fiction.  This suggests the idea,
reminiscent of Pythagoras' views, that Physics could be formulated
with rational numbers only.  We can therefore wonder why real numbers
have been invented and, moreover, used in Physics.  A hypothesis is
that the invention of real numbers is one of the many situations,
where the complexity of an object is increased, so that it can be
apprehended more easily.  (Dowek 2013 \cite{Do13})
\end{quote}
Related comments by Wheeler (1994 \cite[p.~308]{Whe}), Brukner \&
Zeilinger (2005 \cite[p.~59]{BZ}), and others were analyzed by Kanovei
et al.~(2013 \cite[Section~8.4]{KKM}).  See also Jaroszkiewicz (2014
\cite{Ja14}).

If all physical quantities can be entirely parametrized by using
rational numbers, there should be no compelling reason to choose the
real number system as the value range of our probability measures.
However, Easwaran is apparently comfortable with the idealisation of
exploiting a larger number system than the rationals for the value
range of probability measures.  What we argue is that the real numbers
are merely one among possible idealisations that can be used for this
purpose.  For instance, in hyperreal models for infinite sequence of
coin tosses developed by Benci, Bottazzi \& Di Nasso (2013
\cite{BBD}), all events have \emph{hyperrational} probabilities.  This
generalizes both the case of finite sequences of coin tosses, and the
Kolmogorovian model for infinite sequences of coin tosses, where a
real-valued probability is generated by applying Caratheodory's
extension theorem to the rational-valued probability measure over the
cylinder sets.

Given Easwaran's firm belief that ``the function relating credences to
the physical is not so complex that its existence is independent of
Zermelo-Fraenkel set theory" (see his section 8.4, premise 2), it is
surprising to find him suggesting that 
\begin{quote}
the surreal numbers seem more promising as a device for future
philosophers of probability to use (Easwaran 2014
\cite[Appendix~A.3]{Ea14}).
\end{quote}

However, while the construction of the surreals indeed ``is a
simultaneous generalization of Dedekind's construction of the real
numbers and von Neumann's construction of the ordinals", as observed
by Easwaran, it is usually carried out in the Von
Neumann--Bernays--G\"odel set theory (NBG) with Global Choice; see,
for instance, the ``Preliminaries'' section of (Alling 1987
\cite{Al87}).  The assumption of the Global Axiom of Choice is a
strong foundational assumption.


The construction of the surreal numbers can be performed within a
version of NBG that is a conservative extension of ZFC, but does not
need Limitation of Size (or Global Choice).  However, NBG clearly is
not a conservative extension of ZF; and if one wishes to prove certain
interesting features of the surreals one needs an even stronger
version of NBG that involves the Axiom of Global Choice.  Therefore,
the axiomatic foundation that one needs for using the surreal numbers
is at least as strong as the one needed for the hyperreals.


\subsection
{Infinity and infinitesimal: let both pretty severely alone}

At the previous turn of the century, H.~Heaton wrote:
\begin{quote}
I think I know exactly what is meant by the term zero.  But I can have
no conception either of infinity or of the infinitesimal, and I think
it would be well if mathematicians would let both pretty severely
alone (Heaton 1898 \cite[p.~225]{H1898}).
\end{quote}
Heaton's sentiment expresses an unease about a mathematical concept of
which one may have an intuitive grasp%
\footnote{The intuitive appeal of infinitesimals make them an
effective teaching tool.  The pedagogical value of teaching calculus
with infinitesimals was demonstrated in a controlled study by Sullivan
(1976 \cite{Su}).}
but which is not easily formalizable.  Heaton points out several
mathematical inconsistencies or ill-chosen terminology among the
conceptions of infinitesimals of his contemporaries.  This highlights
the brilliant mathematical achievement of a consistent ``calculus''
for infinitesimals attained through the work of Hewitt (1948
\cite{Hew48}), \Los{} (1955 \cite{Lo}), Robinson (1961 \cite{Ro61}),
and Nelson (1977 \cite{N77}), but also of their predecessors like
Fermat, Euler, Leibniz, and Cauchy, as we analyze respectively in
Sections~\ref{fermat}, \ref{lei}, \ref{euler}, and~\ref{cauchy}.

\section{Fermat's adequality}
\label{fermat}

Our interpretation of Fermat's technique is compatible with those by
\Stromholm{} (1968 \cite{Strom}) and Giusti (2009 \cite{Giu}).  It is
at variance with the interpretation by Breger (1994 \cite{Bre94}),
considered by Knobloch (2014 \cite{Kn14}) to have been refuted.

Adequality, or~\parisotes{} (parisot\=es) in the original Greek of
Diophantus,
%
%
is a crucial step in Fermat's method of finding maxima, minima,
tangents, and solving other problems that a modern mathematician would
solve using infinitesimal calculus.  The method is presented in a
series of short articles in Fermat's collected works.
The first article, {\em Methodus ad Disquirendam Maximam et
Minimam\/},
%
%
opens with a summary of an algorithm for finding the maximum or
minimum value of an algebraic expression in a variable~$A$.  For
convenience, we will write such an expression in modern functional
notation as~$f(a)$.
%
%

\subsection{Summary of Fermat's algorithm}
\label{summary}

One version of the algorithm can be broken up into six steps in the
following way:

\begin{enumerate}
\item
Introduce an auxiliary symbol~$e$, and form~$f(a+e)$;
\item
Set {\em adequal\/} the two expressions~$f(a+e) =_{\text{AD}} f(a)$
(the notation ``$=_{\text{AD}}$'' for adequality is ours, not
Fermat's);
\item
\label{cancel}
Cancel the common terms on the two sides of the adequality.  The
remaining terms all contain a factor of~$e$;
\item
\label{divide}
Divide by~$e$ (see also next step);
\item
\label{higher}
In a parenthetical comment, Fermat adds: ``or by the highest common
factor of~$e$";
\item
\label{among}
Among the remaining terms, suppress
%
%
all terms which still contain a factor of~$e$.
%
%
Solving the resulting equation for~$a$ yields the extremum of~$f$.
\end{enumerate}

In modern mathematical language, the algorithm entails expanding the
difference quotient
\[
\frac{f(a+e)-f(a)}{e}
\]
in powers of~$e$ and taking the constant term.%
\footnote{Fermat also envisions a more general technique involving
division by a higher power of~$e$ as in step~\eqref{higher}.
}
The method (leaving aside step~\eqref{higher}) is immediately
understandable to a modern reader as the elementary calculus exercise
of finding the extremum by solving the equation~$f'(a)=0$.  But the
real question is how Fermat understood this algorithm in his own
terms, in the mathematical language of his time, prior to the
invention of calculus by Barrow, Leibniz, Newton, and others.

There are two crucial points in trying to understand Fermat's
reasoning: first, the meaning of ``adequality'' in step (2), and
second, the justification for suppressing the terms involving positive
powers of~$e$ in step~\eqref{among}.  The two issues are closely
related because interpretation of adequality depends on the conditions
on~$e$.  One condition which Fermat always assumes is that~$e$ is
positive.  He did not use negative numbers in his calculations.%
\footnote{This point is crucial for our argument below using the
transverse ray.  Since Fermat is only working with positive values of
his~$e$, he only considers a ray (rather than a full line) starting at
a point of the curve.  The convexity of the curve implies an
inequality, which Fermat transforms into an \emph{adequality} without
giving much explanation of his procedure, but assuming implicitly that
the ray is tangent to the curve.  But a transverse ray would satisfy
the inequality no less than a tangent ray, indicating that Fermat is
relying on an additional piece of geometric information.  His
procedure of applying the defining relation of the curve itself, to a
point on the tangent ray, is only meaningful when the increment~$e$ is
small (see Subsection~\ref{four}).}

Fermat introduces the term \emph{adequality} in \emph{Methodus} with a
reference to Diophantus of Alexandria.  In the third article of the
series, \emph{Ad Eamdem Methodum} (\emph{Sur la M\^{e}me M\'ethode}),
he quotes Diophantus' Greek term~\parisotes, which he renders
following Xylander and Bachet, as \emph{adaequatio} or
\emph{adaequalitas} (see A.~Weil \cite[p.~28]{We84}).

\subsection{Tangent line and convexity of parabola}
\label{four}

Consider Fermat's calculation of the tangent line to the parabola (see
Fermat \cite[p.~122-123]{Fer}).  To simplify Fermat's notation, we
will work with the parabola~$y=x^2$, or
\[
\frac{x^2}{y}=1.
\]
To understand what Fermat is doing, it is helpful to think of the
parabola as a level curve of the two-variable
function~$\frac{x^2}{y}$.

Given a point~$(x,y)$ on the parabola, Fermat wishes to find the
tangent line through the point.  Fermat exploits the geometric fact
that by convexity, a point
\[
(p,q)
\]
on the tangent line lies {\em outside\/} the parabola.  He therefore
obtains an inequality equivalent in our notation to~$\frac{p^2}{q}>1$,
or~$p^2>q$.  Here~$q=y-e$, and~$e$ is Fermat's magic symbol we wish to
understand.  Thus, we obtain
\begin{equation}
\label{41}
\frac{p^2}{y-e}>1.
\end{equation}
At this point Fermat proceeds as follows:
\begin{enumerate}
\item[(i)]
\label{shalosh}
he writes down the inequality~$\frac{p^2}{y-e}>1$, or~${p^2}>{y-e}$;
\item[(ii)]
\label{arba}
he invites the reader to {\em ad\'egaler\/} (to ``adequate'');
\item[(iii)]
\label{ve}
he writes down the
adequality~$\frac{x^2}{p^2}=_{\text{AD}}^{\phantom{I}}\frac{y} {y-e}$;
\item[(iv)] 
he uses an identity involving similar triangles to substitute 
\[
\frac{x}{p}=\frac{y+r}{y+r-e}
\]
where~$r$ is the distance from the vertex of the parabola to the point
of intersection of the tangent to the parabola at~$y$ with the axis of
symmetry,
\item[ {(v)}] he cross multiplies and cancels identical terms on right
and left, then divides out by~$e$, discards the remaining terms
containing~$e$, and obtains~$y=r$ as the solution.%
\footnote{In Fermat's notation~$y=d$,~$y+r=a$. Step (v) can be
understood as requiring the
expression~$\frac{y}{y-e}-\frac{(y+r)^2}{(y+r-e)^2}$ to have a double
root at~$e=0$, leading to the solution~$y=r$ or in Fermat's
notation~$a=2r$.}
\end{enumerate}

What interests us here are steps~(i) and (ii).  How does Fermat pass
from an inequality to an adequality?  Giusti noted that
\begin{quote}
Comme d'habitude, Fermat est autant d\'etaill\'e dans les exemples
qu'il est r\'eticent dans les explications.  On ne trouvera donc
presque jamais des justifications de sa r\`egle des tangentes (Giusti
2009 \cite{Giu}).
\end{quote}
In fact, Fermat provides no explicit explanation for this step.
However, what he does is to apply the defining relation for a curve to
points on the tangent line to the curve.  Note that here the
quantity~$e$, as in~$q=y-e$, is positive: Fermat did not have the
facility we do of assigning negative values to variables.
\Stromholm~notes that Fermat
\begin{quote}
never considered negative roots, and if~$A=0$ was a solution of an
equation, he did not mention it as it was nearly always geometrically
uninteresting (\Stromholm{} 1968 \cite[p.~49]{Strom}).
\end{quote}

Fermat says nothing about considering points~$y+e$ ``on the other
side'', i.e., further away from the vertex of the parabola, as he does
in the context of applying a related but different method, for
instance in his two letters to Mersenne (see \cite[p.~51]{Strom}), and
in his letter to Br\^ulart~\cite{Fer2}.%
\footnote{This was noted by Giusti (2009 \cite{Giu}).}
Now for positive values of~$e$, Fermat's inequality~\eqref{41} would
be satisfied by a {\em transverse ray\/} (i.e., secant ray) starting
at~$(x,y)$ and lying outside the parabola, just as much as it is
satisfied by a tangent ray starting at~$(x,y)$.  Fermat's method
therefore presupposes an additional piece of information, privileging
the tangent ray over transverse rays.  The additional piece of
information is geometric in origin: he applies the defining relation
(of the curve itself) to a point on the tangent ray to the curve, a
procedure that is only meaningful when the increment~$e$ is small.


In modern terms, we would speak of the tangent line being a ``best
approximation'' to the curve for a small variation~$e$; however,
Fermat does not explicitly discuss the size of~$e$.  The procedure of
``discarding the remaining terms'' in step (v) admits of a proxy in
the hyperreal context.  Namely, it is the standard part principle (see
Section~\ref{one}).  Fermat does not elaborate on the justification of
this step, but he is always careful to speak of the \emph{suppressing}
or \emph{deleting} the remaining term in~$e$, rather than setting it
equal to zero.  Perhaps his rationale for suppressing terms in~$e$
consists in ignoring terms that don't correspond to an actual
measurement, prefiguring Leibniz's \emph{inassignable quantities}.
Fermat's inferential moves in the context of his adequality are akin
to Leibniz's in the context of his calculus; see Section~\ref{lei}.

\subsection
{Fermat, Galileo, and Wallis}

While Fermat never spoke of his~$e$ as being \emph{infinitely small},
the technique was known both to Fermat's contemporaries like Galileo
(see Bascelli 2014 \cite{Ba14}, \cite{Ba14b}) and Wallis (see Katz \&
Katz \cite[Section~24]{KK11a}) as well as Fermat himself, as his
correspondence with Wallis makes clear; see Katz, Schaps \& Shnider
(2013 \cite[Section~2.1]{KSS13}).

Fermat was very interested in Galileo's treatise \emph{De motu
locali}, as we know from his letters to Marin Mersenne dated apr/may
1637, 10~august, and 22 october 1638.  Galileo's treatment of
infinitesimals in \emph{De motu locali} is discussed by Wisan (1974
\cite[p.~292]{Wi74}) and Settle (1966 \cite{Se}).

Alexander (2014 \cite{Al14}) notes that the clerics in Rome forbade
the doctrine of the \emph{infinitely small} on 10 august 1632 (a month
before Galileo was put on trial over heliocentrism); this may help
explain why the catholic Fermat might have been reluctant to speak of
the \emph{infinitely small} explicitly.%
\footnote{See a related discussion at
http://math.stackexchange.com/questions/661999/are-infinitesimals-dangerous}

\newcommand\parisoun{{$\pi\acute\alpha\rho\iota\sigma{o}\tilde\upsilon\nu$}}

In a recent text, U.~Felgner analyzes the Diophantus problems which
exploit the method of \parisotes, and concludes that
\begin{quote}
Aus diesen Beispielen wird deutlich, dass die Verben \parisoun{} und
adaequare nicht ganz dasselbe aus\-dr\"u\-cken.  Das griechische Wort
bedeutet, der Gleichheit nahe zu sein, w\"ahrend das lateinische Wort
das Erreichender Gleichheit (sowohl als vollendeten als auch als
unvollendeten Proze{\ss}) ausdr\"uckt (Felgner 2014 \cite{Fel14}).
\end{quote}

Thus, in his view, even though the two expressions have slightly
different meanings, the Greek meaning ``being close to equality'' and
the Latin meaning ``equality which is reached (at the end of either a
finite or an infinite process),'' they both involve approximation.
Felgner goes on to consider some of the relevant texts from Fermat,
and concludes that Fermat's method has nothing to do with differential
calculus and involves only the property of an auxiliary expression
having a double zero:
\begin{quote}
Wir hoffen, deutlich gemacht zu haben, dass die fermatsche ``Methode
der Adaequatio'' gar nichts mit dem Differential-Kalk\"ul zu hat,
sondern vielmehr im Studium des Wertverlaufs eines Polynoms in der
Umgebung eines kritischen Punktes besteht, und dabei das Ziel verfolgt
zu zeigen, dass das Polynom an dieser Stelle eine doppelte Nullstelle
besitzt (ibid.)
\end{quote}
However, Felgner's conclusion is inconsistent with his own textual
analysis which indicates that the idea of approximation is present in
the methods of both Diophantus and Fermat.  As Knobloch (2014
\cite{Kn14}) notes, ``Fermat's method of adequality is not a single
method but rather a cluster of methods.''  Felgner failed to analyze
the examples of tangents to transcendental curves, such as the
cycloid, in which Fermat does \emph{not} study the order of the zero
of an auxiliary polynomial. Felgner mistakenly asserts that in the
case of the cycloid Fermat did not reveal how he thought of the
solution: ``Wie FERMATsich die L\"osung dachte, hat er nicht
verraten.'' (ibid.)  Quite to the contrary, as Fermat explicitly
stated, he applied the defining property of the curve to points on the
tangent line:
\begin{quote}
Il faut donc ad\'egaler (\`a cause de la propri\'et\'e sp\'ecifique de
la courbe qui est \`a consid\'erer sur la tangente) 
\end{quote}
(see Katz et al.~(2013 \cite{KSS13}) for more details).  Fermat's
approach involves applying the defining relation of the curve, to a
point on a \emph{tangent} to the curve.  The approach is consistent
with the idea of approximation inherent in his method, involving a
negligible distance (whether infinitesimal or not) between the tangent
and the original curve when one is near the point of tangency.  This
line of reasoning is related to the ideas of the differential
calculus.  Note that Fermat does not say anything here concerning the
multiplicities of zeros of polynomials.  As Felgner himself points
out, in the case of the cycloid the only polynomial in sight is of
first order and the increment ``$e$'' cancels out.  Fermat correctly
solves the problem by obtaining the defining equation of the tangent.

For a recent study of 17th century methodology, see the article
(Carroll et al.~2013 \cite{CDP}).

\section
{Leibniz's Transcendental law of homogeneity}
\label{lei}

In this section, we examine a possible connection between Fermat's
adequality and Leibniz's Transcendental Law of Homogeneity (TLH).
Both of them enable certain inferential moves that play parallel roles
in Fermat's and Leibniz's approaches to the problem of maxima and
minima.  Note the similarity in titles of their seminal texts:
\emph{Methodus ad Disquirendam Maximam et Minimam} (Fermat, see
Tannery \cite[pp.~133]{Tan}) and \emph{Nova methodus pro maximis et
minimis \ldots} (Leibniz 1684 \cite{Le84} in Gerhardt \cite{Ge50}).

\subsection{When are quantities equal?}

Leibniz developed the TLH in order to enable inferences to be made
between inassignable and assignable quantities.  The TLH governs
equations involving differentials.  H.~Bos interprets it as follows:
\begin{quote}
A quantity which is infinitely small with respect to another quantity
can be neglected if compared with that quantity.  Thus all terms in an
equation except those of the highest order of infinity, or the lowest
order of infinite smallness, can be discarded.  For instance,
\begin{equation}
\label{adeq2}
a+dx =a
\end{equation}
\[
dx+ddy=dx
\]
etc.  The resulting equations satisfy this [\dots] requirement of
homogeneity (Bos 1974 \cite[p.~33]{Bos} paraphrasing Leibniz 1710
\cite[p.~381-382]{Le10b}).
\end{quote}
The title of Leibniz's 1710 text is \emph{Symbolismus memorabilis
calculi algebraici et infinitesimalis in comparatione potentiarum et
differentiarum, et de lege homogeneorum transcendentali}.  The
inclusion of the transcendental law of homogeneity 
%
%
(\emph{lex homogeneorum transcendentalis})
%
%
in the title of the text attests to the importance Leibniz attached to
this law.

The ``equality up to an infinitesimal'' implied in TLH was explicitly
discussed by Leibniz in a 1695 response to Nieuwentijt, in the
following terms:

\begin{quote}
Caeterum \emph{aequalia} esse puto, non tantum quorum differentia est
omnino nulla, sed et quorum differentia est incomparabiliter parva; et
licet ea Nihil omnino dici non debeat, non tamen est quantitas
comparabilis cum ipsis, quorum est differentia (Leibniz 1695
\cite[p.~322]{Le95}) [emphasis added--authors]
\end{quote}
We provide a translation of Leibniz's Latin:
\begin{quote}
Besides, I consider to be \emph{equal} not only those things whose
difference is entirely nothing, but also those whose difference is
incomparably small: and granted that it [i.e., the difference] should
not be called entirely Nothing, nevertheless it is not a quantity
comparable to those whose difference it is.
\end{quote}

\subsection{Product rule}
How did Leibniz use the TLH in developing the calculus? The issue can
be illustrated by Leibniz's justification of the last step in the
following calculation:
\begin{equation}
\label{41c}
\begin{aligned}
d(uv) &= (u+du)(v+dv)-uv\\&=udv+vdu+du\,dv \\ & =udv+vdu.
\end{aligned}
\end{equation}
The last step in the calculation~\eqref{41c} depends on the following
inference:
\[
d(uv)=udv+vdu+dudv \quad \Longrightarrow \quad d(uv)=udv+vdu.
\]
Such an inference is an application of Leibniz's TLH.  In his 1701
text \emph{Cum Prodiisset} \cite[p.~46-47]{Le01c}, Leibniz presents an
alternative justification of the product rule (see Bos
\cite[p.~58]{Bos}).  Here he divides by~$dx$, and argues with
differential \emph{quotients} rather than differentials.  The role
played by the TLH in these calculations is similar to that played by
adequality in Fermat's work on maxima and minima.  For more details on
Leibniz, see Guillaume (2014 \cite{Gu}); Katz \& Sherry (2012
\cite{KS2}), (2013 \cite{KS1}); Sherry \& Katz \cite{SK}; Tho (2012
\cite{Th}).

\section
{Euler's Principle of Cancellation}
\label{euler}

Some of the Leibnizian formulas reappear, not surprisingly, in his
student's student Euler.  Euler's formulas like
\begin{equation}
\label{fo31}
a+dx=a,
\end{equation}
where~$a$ ``is any finite quantity'' (see Euler 1755
\cite[\S\,\S\,86,87]{Eu55}) are consonant with a Leibnizian tradition
as reported by Bos; see formula~\eqref{adeq2} above.  To explain
formulas like~\eqref{fo31}, Euler elaborated two distinct ways
(arithmetic and geometric) of comparing quantities, in the following
terms:

\begin{quote}
Since we are going to show that an infinitely small quantity is really
zero, we must meet the objection of why we do not always use the same
symbol 0 for infinitely small quantities, rather than some special
ones\ldots [S]ince we have two ways to compare them, either
\emph{arithmetic} or \emph{geometric}, let us look at the quotients of
quantities to be compared in order to see the difference.

If we accept the notation used in the analysis of the infinite,
then~$dx$ indicates a quantity that is infinitely small, so that both
$dx =0$ and~$a\,dx=0$, where~$a$ is any finite quantity.  Despite
this, the \emph{geometric} ratio~$a\,dx: dx$ is finite, namely~$a:1$.
For this reason, these two infinitely small quantities,~$dx$
and~$a\,dx$, both being equal to~$0$, cannot be confused when we
consider their ratio.  In a similar way, we will deal with infinitely
small quantities~$dx$ and~$dy$ (ibid., \S\,86, p. 51-52) [emphasis
added--the authors].
\end{quote}
Having defined the arithmetic and geometric comparisons, Euler
proceeds to clarify the difference between them as follows:
\begin{quote}
Let~$a$ be a finite quantity and let~$dx$ be infinitely small.
%
%
The arithmetic ratio of equals is clear: Since~$ndx =0$, we have
\[
a \pm ndx - a = 0.
\]
On the other hand, the geometric ratio is clearly of equals, since
\begin{equation}
\label{32b}
\frac{a \pm ndx}{a} = 1.
\end{equation}
From this we obtain the well-known rule that \emph{the infinitely
small vanishes in comparison with the finite and hence can be
neglected [with respect to it]} \cite[\S 87]{Eu55} [emphasis in the
original--the authors].
\end{quote}


Like Leibniz, Euler considers more than one way of comparing
quantities.  Euler's formula~\eqref{32b} indicates that his geometric
comparison is procedurally identical with the Leibnizian TLH.

To summarize, Euler's geometric comparision of a pair of quantities
amounts to their ratio being infinitely close to a finite quantity, as
in formula~\eqref{32b}; the same is true for TLH.  Note that one
has~$a+dx=a$ in this sense for an appreciable~$a\not=0$, but
\emph{not} for ~$a=0$ (in which case there is equality only in the
\emph{arithmetic} sense).  Euler's ``geometric'' comparison was dubbed
``the principle of cancellation'' in (Ferraro \cite[pp.~47, 48,
54]{Fe04}).

Euler proceeds to present the usual rules of infinitesimal calculus,
which go back to Leibniz, L'H\^opital, and the Bernoullis, such as
\begin{equation}
\label{fo32}
a\,dx^m + b\,dx^n=a\,dx^m
\end{equation}
provided~$m<n$ ``since~$dx^n$ vanishes compared with~$dx^m$''
(\cite[\S\,89]{Eu55}), relying on his ``geometric'' comparison.  Euler
introduces a distinction between infinitesimals of different order,
and directly \emph{computes}%
\footnote{\label{f7}Note that Euler does not ``prove that the
expression is equal to 1''; such \emph{indirect} proofs are a
trademark of the~$(\epsilon, \delta)$ approach.  Rather, Euler
\emph{directly} computes (what would today be formalized as the
\emph{standard part} of) the expression.}
a ratio of the form
\[
\frac{dx\pm dx^2}{dx}=1\pm dx=1
\]
of two particular infinitesimals, assigning the value~$1$ to it
(ibid., \S\,88).  Euler concludes:
\begin{quote}
Although all of them [infinitely small quantities] are equal to 0,
still they must be carefully distinguished one from the other if we
are to pay attention to their mutual relationships, which has been
explained through a geometric ratio (ibid., \S\,89).
\end{quote}
The Eulerian hierarchy of orders of infinitesimals harks back to
Leibniz's work (see Section~\ref{lei}).  Euler's \emph{geometric
comparision}, or ``principle of cancellation'', is yet another
incarnation of the idea at the root of Fermat's adequality and
Leibniz's Transcendental Law of Homogeneity.  For further details on
Euler see Bibiloni et al.~(2006 \cite{BVP}); Bair et al.~(2013
\cite{B11}); Reeder (2013 \cite{Re}).

\section{What did Cauchy mean by ``limit''?}
\label{cauchy}

Laugwitz's detailed study of Cauchy's methodology places it squarely
in the B-track (see Section~\ref{two}).  In conclusion, Laugwitz
writes:
\begin{quote}
The influence of Euler should not be neglected, with regard both to
the organization of Cauchy's texts and, in particular, to the
fundamental role of infinitesimals (Laugwitz 1987
\cite[p.~273]{Lau87}).
\end{quote}
Thus, in his 1844 text \emph{Exercices d'analyse et de physique
math\'ematique}, Cauchy wrote:

\begin{quote}
\ldots si, les accroissements des variables \'etant suppos\'es
infiniment petits, on n\'eglige, vis-\`a-vis de ces accroissements
consid\'er\'es comme infiniment petits du premier ordre, les
infiniment petits des ordres sup\'erieurs au premier, les nouvelles
\'equations deviendront lin\'eaires par rapport aux accroissements
petits des variables.  Leibniz et les premiers g\'eom\`etres qui se
sont occup\'es de l'analyse infinit\'esimale ont appel\'e
\emph{diff\'erentielles} des variables leurs accroissements infiniment
petits, \ldots{} (Cauchy 1844 \cite[p.~5]{Ca44}).
\end{quote}

Two important points emerge from this passage.  First, Cauchy
specifically speaks about \emph{neglecting} (``on n\'eglige'') higher
order terms, rather than setting them equal to zero.  This indicates a
similarity of procedure with the Leibnizian TLH (see
Section~\ref{lei}).  Like Leibniz and Fermat before him, Cauchy does
not set the higher order terms equal to zero, but rather ``neglects"
or discards them.  Furthermore, Cauchy's comments on Leibniz deserve
special attention.

\subsection{Cauchy on Leibniz} 

By speaking matter-of-factly about the infinitesimals of Leibniz
specifically, Cauchy reveals that his (Cauchy's) infinitesimals are
consonant with Leibniz's.  This is unlike the differentials where
Cauchy adopts a different approach.

On page 6 of the same text, Cauchy notes that the notion of derivative
\begin{quote}
repr\'esente en r\'ealit\'e la \emph{limite} du rapport entre les
accrossements infiniment petits et simultan\'es de la fonction et de
la variable (ibid., p.~6) [emphasis added--the authors]
\end{quote}
The same definition of the derivative is repeated on page 7, this time
emphasized by means of italics.  Note Cauchy's emphasis on the point
that the derivative is not a ratio of infinitesimal increments, but
rather the \emph{limit} of the ratio.  

Cauchy's use of the term ``limit" as applied to a ratio of
infinitesimals in this context may be unfamiliar to a modern reader,
accustomed to taking limits of \emph{sequences} of real numbers.  Its
meaning is clarified by Cauchy's discussion of ``neglecting" higher
order infinitesimals in the previous paragraph on page 5 cited above.
Cauchy's use of ``limit" is procedurally identical with the Leibnizian
TLH, and therefore similarly finds its modern proxy as extracting the
standard part out of the ratio of infinitesimals.


On page 11, Cauchy chooses infinitesimal increments~$\Delta s$
and~$\Delta t$, and writes down the equation
\begin{equation}
\label{51b}
\frac{ds}{dt}= \text{lim.} \; \frac{\Delta s}{\Delta t} .
\end{equation}
Modulo replacing Cauchy's symbol ``lim." by the modern one ``st" or
``sh'', Cauchy's formula~\eqref{51b} is identical to the formula
appearing in any textbook based on the hyperreal approach, expressing
the derivative in terms of the standard part function (shadow).

\subsection{Cauchy on continuity}

On page 17 of his 1844 text, Cauchy gives a definition of continuity
in terms of infinitesimals (an infinitesimal~$x$-increment necessarily
produces an infinitesimal~$y$-increment).  His definition is nearly
identical with the italicized definition that appeared on page 34 in
his \emph{Cours d'Analyse} (Cauchy 1821 \cite{Ca21}), 23 years
earlier, when he first introduced the modern notion of continuity.  We
will use the translation by Bradley \& Sandifer (2009 \cite{BS09}).
In his Section~2.2 entitled \emph{Continuity of functions}, Cauchy
writes:
\begin{quote}
If, beginning with a value of~$x$ contained between these limits, we
add to the variable~$x$ an infinitely small increment~$\alpha$, the
function itself is incremented by the difference~$f(x+\alpha)-f(x)$.
\end{quote}
Cauchy goes on to state that 
\begin{quote}
the function~$f(x)$ is a continuous function of~$x$ between the
assigned limits if, for each value of~$x$ between these limits, the
numerical value of the difference~$f(x+\alpha)-f(x)$ decreases
indefinitely with the numerical value of~$\alpha$.
\end{quote}
He then proceeds to provide an italicized definition of continuity in
the following terms:
\begin{quote}
\emph{the function~$f(x)$ is continuous with respect to~$x$ between
the given limits if, between these limits, an infinitely small
increment in the variable always produces an infinitely small
increment in the function itself.}
\end{quote}
In modern notation, Cauchy's definition can be stated as follows.
Denote by~$\overset{\halo}{x}$ the \emph{halo} of~$x$, i.e., the
collection of all points infinitely close to~$x$.  Then~$f$ is
continuous at~$x$ if
\begin{equation}
\label{42}
f\left(\overset{\halo}{x}\right)\subset \overset{\halobig}{f(x)}.
\end{equation}
Most scholars hold that Cauchy never worked with a pointwise
definition of continuity (as is customary today) but rather required a
condition of type~\eqref{42} to hold in a range (``between the given
limits'').  It is worth recalling that Cauchy never gave an~$\epsilon,
\delta$ \emph{definition} of either limit or continuity (though
($\epsilon,\delta)$-type \emph{arguments} occasionally do appear in
Cauchy).  It is a widespread and deeply rooted misconception among
both mathematicians and those interested in the history and philosophy
of mathematics that it was Cauchy who invented the modern
$(\epsilon,\delta)$ definitions of limit and continuity; see, e.g.,
Colyvan \& Easwaran (2008 \cite[p.~88]{CE}) who err in attributing the
formal~$(\epsilon,\delta)$ definition of continuity to Cauchy.  That
this is not the case was argued by B\l{}aszczyk et al.~(2013
\cite{BKS}); Borovik et al.~(2012 \cite{BK}); Katz \& Katz (2011
\cite{KK11b}); Nakane (2014 \cite{Na14}); Tall et al.~(2014
\cite{TK}).

\section{Modern formalisations: a case study}

To illustrate the use of the standard part in the context of the
hyperreal field extension of~$\R$, we will consider the following
problem on divergent integrals.  The problem was recently posed at SE,
and is reportedly due to S.~Konyagin.%
\footnote{http://math.stackexchange.com/questions/408311/improper-integral-diverges}
The solution exploits the technique of a monotone rearrangement~$g$ of
a function~$f$, shown by Ryff to admit a measure-preserving map
$\phi:[0,1]\to[0,1]$ such that~$f=g\circ\phi$.  In general there is no
``inverse''~$\psi$ such that~$g=f\circ\psi$; however, a hyperreal
enlargement enables one to construct a suitable (internal) proxy for
such a~$\psi$, so as to be able to write~$g=\text{st} (f\circ\psi)$;
see formula~\eqref{71b} below.

\begin{theorem}
\label{se}
Let~$f$ be a real-valued function continuous on~$[0,1]$.  Then there
exists a number~$a$ such that the integral
\begin{equation}
\label{61}
\int_0^1\frac 1 {|f(x)-a|}\, dx
\end{equation}
diverges.
\end{theorem}

A proof can be given in terms of a monotone rearrangement of the
function (see Hardy et al.~\cite{HLP}).  We take a decreasing
rearrangement~$g(x)$ of the function~$f(x)$.  If~$f$ is continuous,
then the function~$g(x)$ will also be continuous.  If~$f$ is not
constant on any set of positive measure, one can construct~$g$ by
setting
\begin{equation}
\label{62}
g = m^{-1} \quad \text{where} \quad m(y)=\text{meas}\{x:f(x)>y\}.
\end{equation}

Ryff (1970 \cite{Ry}) showed that there exists a measure-preserving
transformation%
\footnote{However, see Section~\ref{54} for a hyperfinite approach
avoiding measure theory altogether.}
$\phi : [0,1] \to [0,1]$ that relates~$f$ and~$g$ as follows:
\begin{equation}
f(x)=g\circ \phi(x)
\end{equation}
Finding a map~$\psi$ such that~$g(x)=f\circ\psi(x)$ is in general
impossible (see Bennett \& Sharpley \cite[p.~85, example~7.7]{BS} for
a counterexample).
%
%
This difficulty can be circumvented using a hyperfinite rearrangement
(see Section~\ref{54}).  By measure preservation, we have
\[
\int_0^1 |f(x)-a|^{-1}\,dx =\int_0^1 |g(x)-a|^{-1}\,dx
\]
(for every~$a$).%
\footnote{\label{trunc}Here one needs to replace the function
$|f(x)-a|^{-1}$ by the family of its truncations~$\min \left(C,
|f(x)-a|^{-1}\right)$, and then let~$C$ increase without bound.}

To complete the proof of Theorem~\ref{se}, apply the result that every
monotone function is a.e.~differentiable.%
\footnote{In fact, one does not really need to use the result that
monotone functions are a.e.~differentiable.  Consider the convex hull
in the plane of the graph of the monotone function~$g(x)$, and take a
point where the graph touches the boundary of the convex hull (other
than the endpoints~$0$ and~$1$).  Setting~$a$ equal to
the~$y$-coordinate of the point does the job.}
Take a point~$p\in[0,1]$ where the function~$g$ is differentiable.
Then the number~$a=g(p)$ yields an infinite integral~\eqref{61}, since
the difference~$|g(x)-a|$ can be bounded above in terms of a linear
expression.%
\footnote{Namely, for~$x$ near such a point~$p$, we
have~$|g(x)-a|\leq(|g'(p)|+1)|x-p|$,
hence~$\frac{1}{|g(x)-a|}\geq\frac{1}{(|g'(p)|+1)|x-a|}$, yielding a
lower bound in terms of a divergent integral.}
%


\section{A combinatorial approach to decreasing rearrangements}
\label{54}

The existence of a decreasing rearrangement of a function~$f$
continuous on~$[0,1]$ admits an elegant proof in the context of its
hyperreal extension~$^{\ast}\!f$, which we will continue to denote
by~$f$.

We present a combinatorial argument showing that the decreasing
rearrangement obeys the same modulus of uniformity as the original
function.%
\footnote{\label{mou}A function~$f$ on~$[0,1]$ is said to satisfy a
modulus of uniformity~$\mu(n)>0, n\in\N$, if~$\forall n\in\N\;\forall
p,q\in[0,1]
\left(|p-q|\leq\mu(n)\rightarrow|f(p)-f(q)|\leq\frac1n\right)$.}
The argument actually yields an independent construction of the
decreasing rearrangement (see Proposition~\ref{72}) that avoids
recourse to measure theory.  It also yields an ``inverse up to an
infinitesimal,''~$\psi$ (see formula~\eqref{71b}), to the
function~$\phi$ such that~$f=g\circ\phi$.  For a recent application of
combinatorial arguments in a hyperreal framework, see Benci et
al.~(2013 \cite{BBD}).

In passing from the finite to the continuous case of rearrangements,
Bennett and Sharpley \cite{BS} note that 
\begin{quote}
nonnegative sequences~$(a_1,a_2,\ldots,a_n)$ and
$(b_1,b_2,\ldots,b_n)$ are equimeasurable if and only if there is a
permutation
$\sigma$ of~$\{1,2,\ldots,n\}$ such that~$b_i=a_{\sigma(i)}$
for~$i=1,2,\ldots,n$. \ldots{} The notion of permutation is no longer
available in this context [of continuous measure spaces] and is
replaced by that of a ``measure-preserving transformation'' (Bennett
and Sharpley 1988 \cite[p.~79]{BS}).
\end{quote}
We show that the hyperreal framework allows one to continue working
with combinatorial ideas, such as the ``inverse'' function~$\psi$, in
the continuous case as well.

Let~$H\in{}^{\ast}\N\setminus\N$, let~$p_i=\frac{i}{H}$ for
$i=1,2,\ldots,H$.  By the Transfer Principle (see e.g., Davis
\cite{Da77}; Herzberg \cite{Her}; Kanovei \& Reeken \cite{kr}), the
nonstandard domain of internal sets satisfies the same basic laws as
the usual, ``standard'' domain of real numbers and related objects.
Thus, as for finite sets, there exists a permutation~$\psi$ of the
hyperfinite grid
\begin{equation}
\label{grid}
G_H=\{ p_1, \ldots, p_H \}
\end{equation}
by decreasing value of~$f(p_i)$ (here~$f(\psi(p_1))$ is the maximal
value).  We assume that equal values are ordered lexicographically so
that~$ \text{if\;} f(p_i)=f(p_j) \text{\;with\;} i<j \text{\;then\;}
\psi(p_i)<\psi(p_j)$.  Hence we obtain an internal function
\begin{equation}
\label{71b}
\hat g(p_i) = f(\psi(p_i)), \quad i=1,\ldots,H.
\end{equation}
Here~$\hat g$ is (perhaps nonstrictly) decreasing on the grid~$G_H$
of~\eqref{grid}.  The internal sequences~$(f(p_i))$ and~$(\hat
g(p_i))$, where~$i=1,\ldots,H$, are equinumerable in the sense above.

\begin{proposition}
\label{72}
Let~$f$ be an arbitrary continuous function.  Then there is a standard
continuous real function~$g(x)$ such that~$g({\rm st}(p_i))={\rm
st}(\hat g(p_i))$ for all~$i$, where~${\rm st}(y)$ denotes the
standard part of a hyperreal~$y$.
\end{proposition}

\begin{proof}
Let~$g_i=\hat g(p_i)$.  We claim that~$\hat g$ is S-continuous
(microcontinuous), i.e., for each pair~$i,j=1,...,H$, if~$p_i - p_j$
is infinitesimal then so is~$\hat g(p_i) - \hat g(p_j)$.  To prove the
claim, we will prove the following stronger fact:

\begin{quote}
for every~$i<j$ there are~$m<n$ such that~$n-m \leq j-i$ and
$|f(p_m)-f(p_n)| \geq g_i-g_j$.
\end{quote}
The sets~$A=\{k:f(p_k) \geq g_i\}$ and~$B=\{k: f(p_k) \leq g_j\}$ are
nonempty and there are at most~$j-i-1$ points which are not in~$A\cup
B$.  Let~$m \in A$ and~$n \in B$ be such that~$|m-n|$ is minimal.  All
integers between~$m$ and~$n$ are not in~$A\cup B$.  Hence there are at
most~$j-i-1$ such integers, and therefore~$|n-m| \leq j-i$.  By
definition of~$A$ and~$B$, we obtain~$|f(p_n)-f(p_m)| \geq g_i-g_j$,
which proves the claim.  Thus~$\hat g$ is indeed S-continuous.

This allows us to define, for any standard~$x\in [0,1]$, the
value~$g(x)$ to be the standard part of the hyperreal~$g_i$ for any
hyperinteger~$i$ such that~$p_i$ is infinitely close to~$x$, and
then~$g$ is a continuous%
\footnote{The argument shows in fact that the modulus of uniformity of
$g$ is bounded by that of~$f$; see footnote~\ref{mou}.}
and (non-strictly) monotone real function equal to the decreasing
rearrangement~$g=m^{-1}$ of~\eqref{62}.
\end{proof}

The hyperreal approach makes it possible to solve Konyagin's problem
without resorting to standard treatments of decreasing rearrangements
which use measure theory.  Note that the rearrangement defined by the
internal permutation~$\psi$ preserves the integral of~$f$ (as well as
the integrals of the truncations of~$|f(x)-a|^{-1}$), in the following
sense.  The \emph{right-hand} Riemann sums satisfy
\begin{equation}
\label{73}
\sum_{i=1}^{H} f(p_i)\Delta x = \sum_{i=1}^{H} f(\psi(p_i)) \Delta x =
 \sum_{i=1}^{H} \hat g(p_i) \Delta x,
\end{equation}
where~$\Delta x = \frac{1}{H}$.  Thus~$\psi$ transforms a hyperfinite
Riemann sum of~$f$ into a hyperfinite Riemann sum of~$\hat g$.
Since~$\int_0^1 f(x)dx=\st\left(\sum_{i=1}^{H}f(p_i)\Delta x\right)$
and~$g(\st(p_i))=\st\left(\hat g(p_i)\right)$, we conclude that~$f$
and~$g$ have the same integrals, and similarly for the integrals
of~$|f(x)-a|^{-1}$; see footnote~\ref{trunc}.

The first equality in \eqref{73} holds automatically by the transfer
principle even though~$\psi$ is an infinite permutation.  (Compare
with the standard situation where changing the order of summation in
an infinite sum generally requires further justification.)  This
illustrates one of the advantages of the hyperreal approach.

\section{Conclusion}

We have critically reviewed several common misrepresentations of
hyperreal number systems, not least in relation to their alleged
non-constructiveness, from a historical, philosophical, and
set-theoretic perspective.  In particular we have countered some of
Easwaran's recent arguments against the use of hyperreals in formal
epistemology.  A hyperreal framework enables a richer syntax better
suited for expressing proxies for procedural moves found in the work
of Fermat, Leibniz, Euler, and Cauchy.  Such a framework sheds light
on the internal coherence of their procedures which have been often
misunderstood from a whiggish post-Weierstrassian perspective.

\section*{Acknowledgments}

The work of Vladimir Kanovei was partially supported by RFBR grant
13-01-00006.  M.~Katz was partially funded by the Israel Science
Foundation grant no.~1517/12.  We are grateful to Thomas Mormann and
to the anonymous referee for helpful suggestions, and to Ivor
Grattan-Guinness for contributing parts of the introduction.

\end{document}